\documentclass[11pt]{amsart}
\usepackage{amsmath}
\usepackage{amsthm}
\usepackage{amssymb}
\usepackage{amsfonts}
\usepackage{graphicx}
\usepackage{mathrsfs,bm,amscd}
\usepackage{latexsym,enumitem}
\usepackage{color}
\usepackage{marginnote}
\usepackage{hyperref}
\usepackage{caption}
\usepackage{subcaption}
\usepackage{graphicx}
\usepackage{mathrsfs,bm,amscd}

\usepackage[all]{xy}

\numberwithin{equation}{section}
\theoremstyle{plain}

\newtheorem{thm}{Theorem}[section]
\newtheorem{prop}[thm]{Proposition}

\newtheorem{lem}[thm]{Lemma}

\theoremstyle{definition}

\newtheorem*{bprob*}{Bonus problem}

\def\N{\mathbb{N}}

\def\e{\epsilon}

\def\a{\alpha}

\newcommand{\ci}{\mathbf{i}}

\DeclareMathOperator{\diam}{diam}

\newcommand{\cS}{\mathcal{S}}

\newcommand{\Hold}{\mathop\mathrm{H\ddot{o}ld}\nolimits}

\author{Efstathios-K. Chrontsios-Garitsis}
\author{Vyron Vellis}

\thanks{V.V was partially supported by NSF DMS grant 2154918.}

\address{Department of Mathematics\\ The University of Tennessee\\ Knoxville, TN 37966}
\email{echronts@utk.edu}
\address{Department of Mathematics\\ The University of Tennessee\\ Knoxville, TN 37966}
\email{vvellis@utk.edu}

\begin{document}

\title{H\"older spiral arcs}

\begin{abstract}
We establish a quantitative necessary and sufficient condition for a spiral arc to be a H\"older arc. The class of spiral arcs contains spirals studied  by Fraser in \cite{Fra_spirals}, and by Burell-Falconer-Fraser in \cite{BFF}. As an application, we recover the sharp result on the H\"older winding problem for polynomial spirals, initially proved in \cite{Fra_spirals}. Moreover,  we provide a sharp exponent estimate for the H\"older classification of polynomial spirals, which coincides with the corresponding quasiconformal classification estimate, and improve certain exponent bounds on the H\"older classification of elliptical spirals  from \cite{BFF}.
\end{abstract}

\maketitle

\section{Introduction}

Given a continuous function $\phi :[2\pi,+\infty) \to (0,\infty)$ with $\lim_{t\rightarrow \infty} \phi(t)= 0$, we denote by $\mathcal{S}_{\phi}$ the spiral 
\[ \{\phi(t) e^{\ci t}:t\in [2\pi,+\infty)\}\cup\{(0,0)\}.\]
Spirals hold a prominent role in fluid turbulence \cite{foi, vass, vasshunt}, dynamical systems \cite{zub, spirals_ode}, and even certain types of models in mathematical biology \cite{Tyson_spirals, Murray_book}. Moreover, they provide examples of ``non-intuitive" fractal behavior (see \cite{dup}), while they have also been extensively studied due to their unexpected analytic properties. For instance, Katznelson-Nag-Sullivan \cite{unwindspirals} demonstrated the dual nature of spirals $\mathcal{S}_{\phi}$ for decreasing $\phi$, lying in-between smoothness and ``roughness", as well as their connection to certain Riemann mapping questions. The existence of Lipschitz and H\"older parametrizations of certain $\cS_\phi$ has also been  studied by the aforementioned authors in \cite{unwindspirals}, by Fish-Paunescu in \cite{Fish}, and by Fraser in \cite{Fra_spirals}. 

In particular, Fraser in \cite{Fra_spirals} focuses on \textit{polynomial} spirals where $\phi(t)=t^{-p}$, for $p>0$, and shows that $\cS_p:= \mathcal{S}_{\phi}$ is an $\a$-H\"older arc for all $\a \in (0,p)$, with this upper bound on the exponent $\alpha$ being sharp.  In the same paper, Fraser suggests a programme of research focused on determining quantitative conditions under which two sets are bi-H\"older equivalent (see \cite[p.~3254]{Fra_spirals}). Towards this direction,
Burell-Falconer-Fraser \cite{BFF} further studied the elliptical spirals 
$$
\mathcal{S}_{p,q}= \{ t^{-p}\cos t+ \ci t^{-q}\sin t: t\in [2\pi,\infty) \}\cup \{(0,0)\},
$$ and provided bounds on the exponent of H\"older maps between two such spirals. Note that all elliptical spirals $\cS_{p,q}$ can be written in the form $\cS_{\phi}$, for some appropriate function $\phi :[2\pi,+\infty) \to (0,\infty)$.

Motivated by the interest in the regularity and H\"older classification of continuous spirals with no self-intersections, we  
define and study a general class of spirals that contains those studied in \cite{Fra_spirals,BFF}.
Given a continuous $\phi :[2\pi,+\infty) \to (0,\infty)$ with $\lim_{t\rightarrow \infty} \phi(t)= 0$, for all $j\in \N$ set 
$$
\mathcal{S}_{\phi}^j:=\{ \phi(t)e^{\ci t}: t\in [2\pi j, 2\pi (j+1)] \},
$$ and
$$
\phi_j:= \max\{ \phi(t):t\in [2\pi j, 2\pi (j+1)] \}.
$$
We say $\mathcal{S}_\phi$ is \textit{almost circular} if there is $C_\phi>0$ such that $\ell(\mathcal{S}_{\phi}^j)\leq C_\phi \phi_j$, for all $j\in \N$, where $\ell(\mathcal{S}_{\phi}^j)$ denotes the length of $\mathcal{S}_{\phi}^j$.
The main result is a necessary and sufficient quantitative condition that a spiral arc (i.e.~spiral with no self-intersections) with the above property needs to satisfy in order to be a H\"older arc.

\begin{thm}\label{thm: main}
Let $s>1$ and $\mathcal{S}_{\phi}$ be an almost circular spiral arc. Then  $\mathcal{S}_{\phi}$ is a $(1/s)$-H\"older arc if, and only if, $\sum_{n=1}^{\infty}\phi_n^s$ converges.
\end{thm}

In particular, Theorem \ref{thm: main} follows from an even more general result we prove for all spiral arcs $\cS_\phi$ with partitions $\cS_\phi^j$ that are H\"older in a uniform way; see Theorem \ref{thm: Holder pieces}.
The almost circular property is introduced for mainly two reasons. First, quantitative conditions on the H\"older regularity of spirals of the form $\cS_\phi$ would be at least as difficult to establish as those for  graphs of functions $\phi$, which is generally a challenging problem. Second, and as already mentioned, the polynomial spiral arcs from \cite{Fra_spirals} and the elliptical spiral arcs from \cite{BFF} are in fact almost circular, thus generalizing these already interesting classes. As a result, Theorem \ref{thm: main} allows us to recover the sharp exponent for the H\"older winding problem studied in \cite{Fra_spirals}, and to improve the results on the H\"older classification of spirals $\cS_{p,q}$ that were previously established in \cite{BFF} (see Section \ref{sec: Holder exponents improve} for details). In fact, in the case of spirals of the form $\cS_p$, $p>0$, Theorem \ref{thm: main} provides a sharp estimate on their H\"older classification in the following sense.

\begin{thm}\label{thm: main2}
    Let {$0<r\leq p\leq 1$}. There is a $r/p$-H\"older map $f:\cS_p\rightarrow\cS_r$, and a Lipschitz map $g: \cS_r\rightarrow\cS_p$. Moreover, every  map $h:\cS_p\rightarrow\cS_r$ that is $\alpha$-H\"older needs to satisfy $\alpha\leq r/p$.
\end{thm}

This paper is organized as follows. In Section \ref{sec:Hparamcriterion} a characterization of a H\"older arc is established by using the notion of variation of a metric arc. Section \ref{sec: Proof of main thm} contains the proof of Theorem \ref{thm: main}, which uses the aforementioned characterization. In Section \ref{sec: Holder exponents improve} we use Theorem \ref{thm: main} to recover the sharp exponents of H\"older regularity for polynomial spirals $\cS_p$ from \cite{Fra_spirals} and improve the classification estimates for elliptical spirals $\cS_{p,q}$ from \cite{BFF}. In the same section, the proof of Theorem \ref{thm: main2} and further remarks on the relation between the H\"older and the quasiconformal classification problem of spirals are included.

\subsection*{Background and notation}
Let $(X,d_X), (Y,d_Y)$ be non-empty metric spaces. We say that a map $f:X\rightarrow Y$ is \textit{$\alpha$-H\"older (continuous)}, for some $\alpha\in (0,1)$, if there is $C>0$ such that 
$$
d_Y(f(x_1), f(x_2))\leq C d_X(x_1,x_2)^\alpha,
$$ for all $x_1, x_2\in X$. The number $\alpha$ is called the \textit{H\"older exponent} of $f$, and the smallest $C>0$ is the \textit{H\"older semi-norm} of $f$, denoted by $\Hold_{\alpha} f$.

Recall that a metric space $X$ is a  \textit{metric arc} if  there is a homeomorphism $f$ mapping the interval $[0,1]$ onto the space $X$. Given an interval $I\subset [0,1]$, we say that $f(I)$ is a \textit{subarc} of $X$. Furthermore, if the interval $I$ has endpoints $a,b\in [0,1]$, we say that $f(I\setminus\{a,b\})$ is the \textit{interior} of the subarc $f(I)$. If $X$ and $f(I)$ have distinct interiors, we say that $f(I)$ is a \textit{proper subarc} of $X$. If $f$ is $\alpha$-H\"older for some $\alpha\in (0,1)$, we say that $X$ is an \textit{$\alpha$-H\"older arc}. 

For $s>0$, $\epsilon>0$, and a subset $E$ of the metric space $X$, the \textit{$s$-dimensional $\epsilon$-approximate Hausdorff measure} of $E$ is defined as
$$
\mathcal{H}_\epsilon^s(E)= \inf 
\left\{ \sum_i (\diam U_i)^s : \{ U_i\}\text{\  countable cover of}
\,\,E \,\,\text{with} \,\, \diam U_i\leq \epsilon  \right\}.
$$ The \textit{$s$-dimensional Hausdorff (outer) measure} of $E$ is the limit
$$
\mathcal{H}^s(E)=\lim_{\epsilon \to 0} \mathcal{H}_\epsilon^s(E).
$$

\section{The $s$-variation and $(1/s)$-H\"older rectifiability}\label{sec:Hparamcriterion}

A classification of Lipschitz arcs (and even more general Lipschitz curves) was given by Wa\.{z}ewski in 1927: an arc is a Lipschitz arc if and only if $\mathcal{H}^1(X)<\infty$, and if $X$ is a Lipschitz arc, then there exists a $(2\mathcal{H}^1(X))$-Lipschitz parameterization $f:[0,1]\to X$. See \cite[Theorem 4.4]{AO-curves} for the proof. An analogous result for $\frac1{s}$-H\"older arcs does not exist as there is no connection between $\frac1{s}$-H\"older parameterizations and $\mathcal{H}^s$; see for example \cite[Proposition 9.8]{BNV}. In this section we characterize $\frac1{s}$-H\"older arcs in terms of the \emph{$s$-variation}, the analogue of $\mathcal{H}^1$ from the Lipschitz arc case.

A \emph{partition} of a metric arc $(X,d_X)$ is a finite collection of subarcs $\mathcal{P} = \{X_1,\dots,X_n\}$ with disjoint interiors, and with their union being equal to $X$. For $s\geq 1$, define the \emph{$s$-variation} of $X$ by
\begin{equation}\label{eq:mass}
\|X\|_{s\text{-var}} := \sup_{\mathcal{P}} \sum_{X' \in \mathcal{P}} (\diam{X'})^s  \in [0,+\infty],
\end{equation}
with the supremum being over all partitions of $X$.

For $s\geq 1$, we set $H_{1/s}(X)$ to be the infimum of all constants $H>0$ for which there is a surjection $f:[0,1] \to X$ satisfying
\[ d_X(f(x),f(y)) \leq H |x-y|^{1/s}, \quad \text{for all $x,y \in [0,1]$.}\]
If no such $f$ exists, then $H_{1/s}(X)=\infty$.

The relation between quantities $\|X\|_{s\text{-var}}$ and $H_{1/s}(X)$, and the existence of H\"older parameterizations is given in the next proposition. 

\begin{prop}\label{prop:mass}
Let $X$ be a metric arc and $s\geq1$. Then
$$
\|X\|_{s\text{-var}}= H_{1/s}(X)^s.
$$
In particular, $X$ is a $(1/s)$-H\"older arc if, and only if, $\|X\|_{s\textup{-var}}<\infty$.
\end{prop}

The H\"older regularity of a metric arc $X$ has been closely tied before to variation notions defined for continuous maps $g:[0,1]\rightarrow X$ (see \cite[Definitions 1.1, 5.1]{Var_book}). Specifically, given a homeomorphism $f:[0,1]\rightarrow X$, the methods leading to the proof of \cite[Proposition 5.14]{Var_book} could be applied to $f$ and yield Proposition \ref{prop:mass}. However, it should be noted that the statement of \cite[Proposition 5.14]{Var_book} is dependent on each given continuous $g:[0,1]\rightarrow X$.  Since our definition of $s$-variation is intrinsically more geometric and does not depend on any given parametrization of $X$, we include the proof in this context.

 
For the proof of Proposition \ref{prop:mass} we require several lemmas. 

\begin{lem}\label{lem:mass1}
Let $X$ be a metric arc and $s\geq1$. 
\begin{enumerate}
\item If $X'$ is a subarc of $X$, then $\|X'\|_{s\textup{-var}} \leq \|X\|_{s\textup{-var}}$. 
\item If $\|X\|_{s\textup{-var}}<\infty$ and $X'$ is a proper subarc of $X$, then $\|X'\|_{s\textup{-var}} < \|X\|_{s\textup{-var}}$.
\item We have $\|X\|_{s\textup{-var}} \geq \max\{\mathcal{H}^s(X),(\diam{X})^s\}$.
\item If $X=X_1\cup\cdots\cup X_n$ is a partition of $X$ into subarcs, then
\[\|X\|_{s\textup{-var}} \geq \|X_1\|_{s\textup{-var}} + \cdots + \|X_n\|_{s\textup{-var}}.\]
\end{enumerate}
\end{lem}

\begin{proof}
Property (1) is immediate from the definition. 

For (2), assume that $\|X\|_{s\textup{-var}}<\infty$ and that $X'$ is a proper subarc of $X$. Let $Y$ be a subarc of $X$ that intersects with $X'$ only at an endpoint. Let also $X_1,\dots,X_n$ be a partition of $X'$ such that
\[ \|X'\|_{s\textup{-var}} \leq \sum_{i=1}^n (\diam{X_i})^s + \tfrac12 (\diam{Y})^s.\]
Then,
\[ \|X\|_{s\textup{-var}} \geq \|X'\cup Y\|_{s\textup{-var}} \geq \sum_{i=1}^n (\diam{X_i})^s + (\diam{Y})^s > \|X'\|_{s\textup{-var}}.\]

For (3), note first that $\mathcal{P}=\{X\}$ is a partition of $X$ so 
$$\|X\|_{s\text{-var}} \geq (\diam{X})^s.$$
To show that $\|X\|_{s\text{-var}} \geq \mathcal{H}^s(X)$, fix $\e>0$ and let $\mathcal{P}$ be a partition of $X$ such that $\diam{X'} < \e$ for all $X'\in\mathcal{P}$. Then,
\[ \mathcal{H}^s_{\e}(X) \leq \sum_{X'\in\mathcal{P}} (\diam{X'})^s \leq   \|X\|_{s\text{-var}}.\]
Letting $\e$ go to $0$, we obtain the desired inequality.

The proof of (4) follows immediately from the definition of $s$-variation.
\end{proof}

\begin{lem}\label{lem:masscont}
Let $s\geq1$, let $X$ be a metric arc with $\|X\|_{s\textup{-var}} < \infty$, and let $f:[0,1]\to X$ be a homeomorphism. Then the function $t \mapsto \|f([0,t])\|_{s\textup{-var}}$ is continuous.
\end{lem}

\begin{proof}
The continuity of the function $t \mapsto \|f([0,t])\|_{s\textup{-var}}$ follows by \cite[Proposition 5.8]{Var_book} applied to the given homeomorphism $f$.
\end{proof}

We are now ready to prove Proposition \ref{prop:mass}.

\begin{proof}[{Proof of Proposition \ref{prop:mass}}] 
Suppose first that $H_{1/s}(X)<\infty$. Then, there exists a $(1/s)$-H\"older homeomorphism $f:[0,1] \to X$ with $H_{1/s}(X)\leq \Hold_{1/s}{f}$. Fix arbitrary $H > \Hold_{1/s}{f}$ and $\e\in (0,1)$. 

Let $X_1,\dots,X_n$ be a partition of $X$ and for each $i\in\{1,\dots,n\}$ let $I_i = f^{-1}(X_i)$. Then $\{I_1,\dots,I_n\}$ is an interval partition of $[0,1]$. For each $i\in\{1,\dots,n\}$, let $x_i,y_i \in I_i$ such that 
\[ d_X(f(x_i) - f(y_i)) \geq (1-\e)\diam{f(I_i)}\]
and denote by $J_i$ the interval in $I_i$ with endpoints $x_i,y_i$. 
By choice of $x_i, y_i$ and H\"older continuity of $f$, we have
\[ (\diam{f(I_i)})^s \leq (1-\e)^{-s} d_X(f(x_i) , f(y_i))^s \leq  (1-\e)^{-s}H^s\diam{J_i}.\]
 Since $I_i=f^{-1}(X_i)$, the above implies
\[ \sum_{i=1}^{n} (\diam{X_i})^s \leq (1-\e)^{-s}H^s\sum_{i=1}^{n} \diam{J_i} \leq (1-\e)^{-s}H^s.\]
Taking supremum over all partitions and letting $\e\to 0$ we obtain $\|X\|_{s\text{-var}} \leq H^s$. Since $H > \Hold_{1/s}{f}$ is arbitrary, and $\Hold_{1/s}{f}$ can be chosen as close to $H_{1/s}(X)$ as necessary, this implies that $\|X\|_{s\text{-var}} \leq H_{1/s}(X)^s$.

For the converse,  let  $f:[0,1] \to X$ be a homeomorphism, $s\geq 1$, and assume that $\|X\|_{s\text{-var}} < \infty$. For each $x\in[0,1]$ define 
\[ \psi(x) = \frac{\|f([0,x])\|_{s\text{-var}}}{\|X\|_{s\text{-var}}}.\]
By Lemma \ref{lem:mass1}, $\psi$ is an increasing function from $[0,1]$ into $[0,1]$ and, by Lemma \ref{lem:masscont}, $\psi$ is continuous. Since $\psi(0) = 0$ and $\psi(1)=1$, it follows that $\psi$ is surjective. By Lemma \ref{lem:mass1}(2), it follows that $\psi$ is in fact a homeomorphism. This allows for the definition of $F=f\circ \psi^{-1}:[0,1] \to X$. 

It remains to show that $F$ is $(1/s)$-H\"older. Let $0 \leq x<y\leq 1$ and let $0\leq x'<y'\leq 1$ be such that $\psi(x') = x$ and $\psi(y') = y$. Then,
\begin{align*}
d_X(F(x), F(y))^s &= d_X(f(x'),f(y'))^s\\ 
&\leq \|f([x',y'])\|_{s\text{-var}}\\ 
&\leq \|f([0,y'])\|_{s\text{-var}} - \|f([0,x'])\|_{s\text{-var}} \\
&= \|X\|_{s\text{-var}} |\psi(x')-\psi(y')|\\ 
&= \|X\|_{s\text{-var}}|x-y|.
\end{align*}
Therefore, $F$ is $(1/s)$-H\"older continuous with $\Hold_{1/s}{F} \leq \|X\|_{s\text{-var}}^{1/s}$. As a result, 
$$
H_{1/s}(X)\leq \Hold_{1/s}{F}\leq \|X\|_{s\text{-var}}^{1/s}<\infty,
$$ which allows to apply the already proved opposite direction and yields $H_{1/s}(X)=\|X\|_{s\text{-var}}^{1/s}$ as needed. This equality and the definition of $H_{1/s}(X)$ complete the proof.
\end{proof}


\section{H\"older rectifiability of spiral arcs}\label{sec: Proof of main thm}

Fix for the rest of this section a spiral arc $\mathcal{S}_{\phi}$, and recall that for $j\in \N$, we set 
$$\cS_\phi^j=\{ \phi(t)e^{\ci t}: t\in [2\pi j, 2\pi (j+1)] \}.$$ Note that since $\cS_\phi$ is an arc, the sequence $(\diam \mathcal{S}_{\phi}^j)_{j\in \N}$ is decreasing. We prove the following result for general spiral arcs.

\begin{thm}\label{thm: Holder pieces}
Let $s\geq 1$. Then $\mathcal{S}_{\phi}$ is a $\frac1{s}$-H\"older arc if and only if $\sum_{j=1}^{\infty}(H_{1/s}(\mathcal{S}_{\phi}^j))^s$ converges.
\end{thm}

Before proving Theorem \ref{thm: Holder pieces}, let us first give the proof of Theorem \ref{thm: main}.

\begin{proof}[{Proof of Theorem \ref{thm: main}}]
Suppose that $\mathcal{S}_{\phi}$ is almost circular. Fix $j\in\N$ and $s\geq 1$. On the one hand, if $h_j : [0,1] \to \mathcal{S}_{\phi}^j$ is the constant-speed Lipschitz parametrization map, 
then for all $x,y \in [0,1]$,
\[ |h_j(x)-h_j(y)| \leq \ell(\mathcal{S}_{\phi}^j)|x-y| \leq \ell(\mathcal{S}_{\phi}^j)|x-y|^{1/s} \leq C_\phi \phi_j |x-y|^{1/s}. \]
On the other hand, if $h : [0,1] \to \mathcal{S}_{\phi}^j$ is $\frac1{s}$-H\"older with H\"older constant $H$, then there exist $x,y \in [0,1]$ such that $|h(x)-h(y)| = \diam{\mathcal{S}_{\phi}^j}$ which gives
\[ H \geq H|x-y|^{1/s} \geq |h(x)-h(y)| = \diam{\mathcal{S}_{\phi}^j} \geq \phi_j.\]
Note that the last inequality above follows by the fact that $\mathcal{S}_\phi^j$ is closed, the fact that $\phi$ is strictly decreasing, and an application of the triangle inequality. Hence, $H_{1/s}(\mathcal{S}_{\phi}^j) \simeq \phi_j$, and the statement follows directly from Theorem \ref{thm: Holder pieces}.
\end{proof}

We now turn to the proof of Theorem \ref{thm: Holder pieces}. 


\begin{proof}[{Proof of Theorem \ref{thm: Holder pieces}}]
Fix $s\geq 1$. We show that
\begin{equation}\label{eq:Holderpieces}
\frac12\|\mathcal{S}_{\phi}\|_{s\text{-var}} \leq \sum_{j=1}^{\infty} (H_{1/s}(\mathcal{S}_{\phi}^j))^s \leq \|\mathcal{S}_{\phi}\|_{s\text{-var}}
\end{equation}
and Theorem \ref{thm: Holder pieces} follows directly from Proposition \ref{prop:mass}.

For the upper bound of \eqref{eq:Holderpieces} fix $k\in\N$ and note that $\mathcal{S}_{\phi}^1,\dots, \mathcal{S}_{\phi}^k, \mathcal{S}_{\phi}\setminus \bigcup_{j=1}^k\mathcal{S}_{\phi}^j$ is a partition of $\mathcal{S}_{\phi}$. By Proposition \ref{prop:mass} and Lemma \ref{lem:mass1}(4),
\[ \sum_{j=1}^{k}(H_{1/s}(\mathcal{S}_{\phi}^j))^s = \sum_{j=1}^{k}\|\mathcal{S}_{\phi}^j\|_{s\text{-var}} \leq \|\mathcal{S}_{\phi}\|_{s\text{-var}}. \]

For the lower bound of \eqref{eq:Holderpieces}, let $X_1, \dots,X_n$ be a partition of $\mathcal{S}_{\phi}$, where the arcs $X_j$ are enumerated according to the orientation of $\mathcal{S}_{\phi}$, with $0 \in X_n$. Let $k\in \N$ be the maximal index so that 
$$
X_n \subset \overline{\mathcal{S}_{\phi}\setminus(\mathcal{S}_{\phi}^1\cup \dots \cup \mathcal{S}_{\phi}^{k-1})}.
$$
We consider three subsets $P_1, P_2, P_3$ of the indices set $\{1,\dots,n\}$, and for indices in each $P_i$ we prove corresponding estimates.

\emph{Estimate 1:} 
Set $P_1=\{X_n\}$. By definition of $\mathcal{S}_{\phi}^k$ and decreasing property of their diameters, by Proposition \ref{prop:mass} and by Lemma \ref{lem:mass1} we have 
\[ (\diam{X_n})^s \leq (\diam \mathcal{S}_{\phi}^k)^s \leq \|\mathcal{S}_{\phi}^k\|_{s\text{-var}} = (H_{1/s}(\mathcal{S}_{\phi}^k))^s.\]

\emph{Estimate 2:} Let
$$
L=\{ l\in\{1,\dots, k\}: \, \text{there exists }j\in \{1,\dots, n\} \,\, \text{such that} \,\, X_j\subset \cS_\phi^l \},
$$
and set
$$
P_2=\{ j\in \{1,\dots, n\}: \, \text{there exists } l\in L\,\, \text{such that}\,\, X_j\subset \cS_\phi^l\}.
$$
Given $l\in L$,  let $\{j_1^l,\dots,j_m^l\}$ be a maximal set of indices in $P_2$ such that $X_{j_1^l},\dots, X_{j_m^l}\subset \cS_\phi^l$. Then, by Lemma \ref{lem:mass1}, and by Proposition \ref{prop:mass} we have
\begin{align*}
(\diam{X_{j_1^l}})^s + \cdots + (\diam{X_{j_m^l}})^s &\leq \| X_{j_1^l}\|_{s\text{-var}} + \cdots + \| X_{j_m^l}\|_{s\text{-var}} \\
&\leq \|X_{j_1^l}\cup\cdots\cup X_{j_m^l}\|_{s\text{-var}}\\
&\leq \|\cS_\phi^l\|_{s\text{-var}}\\
&= (H_{1/s}(\mathcal{S}_{\phi}^l))^s.
\end{align*}

\emph{Estimate 3:} Finally, set $P_3=\{1,\dots, n\}\setminus(P_1\cup P_2)$. For each $j \in P_3$, there exists a minimal $k_j \in \N$ such that $X_j \cap \mathcal{S}_{\phi}^{k_j} \neq \emptyset$. Moreover, if $j,j'$ are two distinct such indices, then $k_j \neq k_{j'}$. For any  $j \in P_3$, by the decreasing property of $(\diam \mathcal{S}_{\phi}^m)_{m\in \N}$, we have 
\begin{align*}
(\diam{X_j})^s \leq (\diam{\bigcup_{m=k_j}^{\infty}\mathcal{S}_{\phi}^m})^s \leq (\diam \mathcal{S}_{\phi}^{k_j})^s \leq \|\mathcal{S}_{\phi}^{k_j}\|_{s\text{-var}} = (H_{1/s}(\mathcal{S}_{\phi}^{k_j}))^s.
\end{align*}

Putting the three estimates together, we get
\[ \sum_{j=1}^{n} (\diam{X_j})^s= \sum_{i=1}^3 \sum_{j\in P_i}(\diam{X_j})^s\leq 2\sum_{n=1}^{\infty}  (H_{1/s}(\mathcal{S}_{\phi}^{n}))^s,\]
since the term $(H_{1/s}(\mathcal{S}_{\phi}^{k}))^s$ may appear at most two times, through the sums over $j$ in $P_1$ and $P_2$, and the rest of the terms may appear at most twice through the sums over $P_2$ and $P_3$.
Since the partition $\{X_1,\dots, X_n\}$ is arbitrary, it follows that 
\[ \|\mathcal{S}_{\phi}\|_{s\text{-var}} \leq 2\sum_{n=1}^{\infty}  (H_{1/s}(\mathcal{S}_{\phi}^{n}))^s. \qedhere\]
\end{proof}

\section{H\"older exponents between spirals}\label{sec: Holder exponents improve}

Suppose $0<p\leq q$, and define the spiral arc
$$
\mathcal{S}_{p,q}= \{ t^{-p}\cos t+ \ci t^{-q}\sin t: t\in [2\pi,\infty) \}\cup \{(0,0)\}.
$$
Burrell, Falconer, and Fraser in \cite[Theorems 2.9, 2.11]{BFF} gave the following upper bounds on the H\"older exponent for maps between such spirals.
\begin{thm}[\cite{BFF}]\label{thm: Falconer}
    Suppose $f:\mathcal{S}_{p,q}\rightarrow \mathcal{S}_{r,s}$ is $\alpha$-H\"older, with $r\leq 1$. If $p> 1$, then
    \begin{equation}\label{eq: Falc_p>1}
        \alpha\leq \frac{1+s}{2+s-r}.
    \end{equation}
    Otherwise, if $p\leq 1$, then 
    \begin{equation}\label{eq: Falc_p<=1}
        \alpha\leq \frac{p+q+r+s-pr+qs}{(2+s-r)(p+q)}.
    \end{equation}
\end{thm}

For $p\leq 1$, Theorem \ref{thm: main} also provides bounds on the H\"older exponent of maps between such spirals in an implicit way. Recall that there are functions $\phi, \psi:[2\pi,+\infty)\rightarrow(0,\infty)$ that tend to $0$ as $t\rightarrow \infty$ with $\mathcal{S}_{p,q}=\mathcal{S}_\phi$ and $\mathcal{S}_{r,s}=\mathcal{S}_\psi$. Note that it is non-trivial to explicitly determine $\phi$ and $\psi$, due to the implicit relation between arguments of 
$$
z_t=t^{-p}\cos t+ \ci t^{-q}\sin t \in \cS_\phi,
$$ 
for some $t\geq 2\pi$, and the modulus $|z_t|$. Namely, while the distance of $z_t$ from $0$ is indeed just $|z_t|$, the value $t$ is not always an argument of $z_t$, which makes the naive approach of choosing $\phi(t)=|z_t|$ generally incorrect for $p\neq q$. However, it is easier to determine $\phi(t_k)$ and $\psi(t_k)$ at $t_k= k\pi/2$, for integers $k\geq 4$, which is enough to imply that $\mathcal{S}_\phi,\mathcal{S}_\psi$ are almost circular. This can also be seen through the relation of these spirals to the corresponding concentric ellipses (see \cite[p.~7]{BFF}). Moreover, by the aforementioned values $\phi(t_k)$ and $\psi(t_k)$, we conclude that
$$
\phi_j= (2\pi j)^{-p}, \,\,\, \psi_j= (2\pi j)^{-r},
$$ for all $j\in \N$. Suppose that $h:\mathcal{S}_{p,q}\rightarrow \mathcal{S}_{r,s}$ is $\alpha$-H\"older. By Theorem \ref{thm: main}, there is a $\beta$-H\"older map $g:[0,1]\rightarrow \cS_{p,q}$, for all $\beta<p$. Thus, the map $h\circ g:[0,1]\rightarrow \cS_{r,s}$ is $\alpha \beta$-H\"older, which by Theorem \ref{thm: main}, and the fact that $\beta$ can be as close to $p$ as needed, implies that
\begin{equation}\label{eq: improve Falc}
    \alpha\leq  \frac{r}{p}.
\end{equation} 

For polynomial spirals studied in \cite{BFF}, the above bound is always an improvement upon Theorem \ref{thm: Falconer}. In particular, for $p=q$ and $r=s$ the bound \eqref{eq: Falc_p<=1} implies that 
$$
\alpha\leq \frac{p+r}{2p}.
$$
In the case where $p\geq r$, the exponent bound \eqref{eq: improve Falc} is indeed an improvement on the above bound stemming from Theorem \ref{thm: Falconer}, i.e., an improvement on \cite[Corollary 2.12]{BFF}. In fact, this is a sharp bound in a more general sense, as stated in Theorem \ref{thm: main2}, which we are ready to prove.
\begin{proof}[Proof of Theorem \ref{thm: main2}]
    Let $0<r\leq p$. If $h:\cS_p\rightarrow\cS_r$ is $\alpha$-H\"older, then is has already been shown that $\alpha\leq r/p$ in \eqref{eq: improve Falc}.

    The desired maps between the spirals $\cS_p, \cS_r$ are in fact appropriate radial stretch maps. In particular, the map $f:\mathbb{C}\rightarrow\mathbb{C}$ defined by 
    $$f(z)=|z|^{\frac{r}{p}-1}z$$ 
    for all $z\neq0$, and $f(0)=0$, is $r/p$-H\"older (see, for instance, \cite[p.~49]{Vais} and \cite[Corollary 3.10.3]{Astala_book}) and satisfies $f(\cS_p)= \cS_r$. Moreover, define the map $g:\mathbb{C}\rightarrow\mathbb{C}$ by  $g(0)=0$ and
    $$
    g(z)=|z|^{\frac{p}{r}-1}z,
    $$for all $z\neq0$. This map satisfies $g(\cS_r)=\cS_p$ and is Lipschitz, due to the derivative being bounded on the closed disk $D(0,(2\pi)^{-r})$. This completes the proof.
\end{proof}

Despite the restriction to polynomial spirals for sharpness, our  bound \eqref{eq: improve Falc} can be shown to be an improvement upon Theorem \ref{thm: Falconer} even for elliptical spirals with specific values of $p\neq q,r\neq s$. In particular, if $0<r<p\leq 1$, then we have the strict inequality
$$
\frac{r}{p}<\frac{p+r}{2p}.
$$ This gap between the two values, along with the continuity of the expression 
$$
\frac{p+q+r+s-pr+qs}{(2+s-r)(p+q)}
$$ in $q$ and $s$, implies that for $q, s$ close enough to $p, r$, respectively, our bound \eqref{eq: improve Falc} is an improvement upon \eqref{eq: Falc_p<=1} for elliptical, non-polynomial spirals as well. For instance, if $p=3/5$, $r=s=1/2$, we have that the exponent bound \eqref{eq: improve Falc} is an improvement upon \eqref{eq: Falc_p<=1} for certain values of $q$ close to $3/5$, i.e.,
$$
\frac{r}{p}=\frac{5}{6}< \frac{13+15q}{12+20q}= \frac{p+q+r+s-pr+qs}{(2+s-r)(p+q)},
$$for all $q\in (3/5,9/5)$. However, it is clear that this is not generally the case, as can be seen by choosing any $q>9/5$ in the above inequality.

It should also be noted that the improved bound \eqref{eq: improve Falc} in the context of H\"older classification for spirals $\cS_p$, $\cS_r$ coincides with the sharp bound in the quasiconformal classification problem resolved by Tyson and the first author in \cite{CG_Tyson}. In particular, by \cite[Theorem 1.1]{CG_Tyson}, two spirals $\cS_p$ and $\cS_r$ are $K$-quasiconformally equivalent if, and only if, $K\geq p/r$ (see \cite{Vais} for more details on quasiconformal mappings). 
It is quite interesting that in the case of these spirals, the sharp exponent bound in the H\"older classification programme suggested by Fraser in \cite{Fra_spirals} is essentially attained from  the sharp dilatation bounds for the corresponding quasiconformal classification study. This motivates further the question of whether and under what conditions resolving the quasiconformal classification problem for two objects results in the resolution of the corresponding H\"older classification problem. We refer the interested reader to the discussion in \cite[Section 5]{CG_Conc} for more details and related results in higher dimensions.

\end{document}